\documentclass[10pt,twoside]{article}
\usepackage{amssymb}
\usepackage{amsmath}
\usepackage{latexsym}
 \newcommand{\COM}[1]{}
\renewcommand{\theequation}{\arabic{section}.\arabic{equation}}

\newtheorem{theorem}{Theorem}[section]

\newtheorem{corollary}{Corollary}[section]
\newtheorem{definition}{Definition}[section]
\newtheorem{remark}{\normalfont\scshape Remark}[section]

\newenvironment{proof}{\noindent\textsc{Proof.\/}}{}
\newenvironment{pf}[1]{\noindent\textsc{Proof of #1.\/}}{}

\newcommand{\subj}[2]{\textsf{AMS 2000 subject classifications.}
Primary {#1}; Secondary {#2}.\newline}
\newcommand{\key}[1]{\textsf{Keywords and phrases.} {#1}.\newline}
\newcommand{\abb}[1]{\textsf{Abbreviated title.} {#1}.}

\newcommand{\fot}[5]{\renewcommand\thefootnote{}
\footnotetext{\parindent=0.0mm \vskip-3mm
\subj{#1}{#2}\key{#3}\abb{#4}
\newline\textsf{Date.} \date{\today}}}

\COM{

}

\def\vsb{\hfill$\Box$}

\newcommand{\bulq}{$\bullet$\quad}

\newcommand{\be}{\begin{equation}}
\newcommand{\ee}{\end{equation}}
\newcommand{\bea}{\begin{eqnarray}}
\newcommand{\eea}{\end{eqnarray}}
\newcommand{\beaa}{\begin{eqnarray*}}
\newcommand{\eeaa}{\end{eqnarray*}}
\newcommand{\beal}{\begin{aligned}}
\newcommand{\eeal}{\end{aligned}}

\newcommand{\var}{\mathrm{Var\,}}

\newcommand{\ifff}{\ttt{$\Longleftrightarrow$}}

\newcommand{\bn}{\mathbf n}

\newcommand{\bk}{\mathbf k}

\newcommand{\sumk}{\sum^n_{k=1}}
\newcommand{\sumin}{\sum_{n=1}^\infty}

\newcommand{\ttt}[1]{\quad\mbox{ #1}\quad}

\newcommand{\asto}{\stackrel{a.s.}{\to}}
\newcommand{\pto}{\stackrel{p}{\to}}

\newcommand{\nifi}{n\to\infty}

\newcommand{\iid}{i.i.d.\ }

\newcommand{\amaa}{\frac{\alpha}{1-\alpha}}
\newcommand{\bnbb}{\frac{\beta}{1-\beta}}
\newcommand{\amaaa}{\frac{1-\alpha}{\alpha}}
\newcommand{\bnbbb}{\frac{1-\beta}{\beta}}

\newcommand{\gmgg}{\frac{\gamma}{1-\gamma}}
\newcommand{\akame}{A_k^{\alpha-1}}

\newcommand{\albme}{A_l^{\beta-1}}
\newcommand{\amkame}{A_{m-k}^{\alpha-1}}
\newcommand{\ankame}{A_{n-k}^{\alpha-1}}

\newcommand{\anlbme}{A_{n-l}^{\beta-1}}

\newcommand{\ama}{A_m^{\alpha}}
\newcommand{\ana}{A_n^{\alpha}}
\newcommand{\anb}{A_n^{\beta}}
\newcommand{\ssum}{\sum_{k,l=1}^{m,n}}
\newcommand{\ssumnull}{\sum_{k,l=0}^{m,n}}
\newcommand{\sumsum}{\sum_{m,n}\,\sum_{k,l=1}^{m,n}}
\newcommand{\sumsumnull}{\sum_{m,n}\,\sum_{k,l=0}^{m,n}}
\newcommand{\xkl}{X_{k,l}}
\newcommand{\etakl}{\eta_{k,l}}
\newcommand{\smn}{S_{m,n}}

\begin{document}
\date{}
\title{\textsf{Ces\`aro summation for random fields}}
\author{Allan Gut\\Uppsala University \and Ulrich Stadtm\"uller\\
University of Ulm}
\maketitle

\begin{abstract}\noindent
Various methods of summation for divergent series of real numbers
have been generalized to analogous results for sums of \iid random
variables. The natural extension of results corresponding to Ces\`aro
summation amounts to proving almost sure convergence of the Ces\`aro
means.  In the present paper we extend such results as well as weak
laws and results on complete convergence to random fields, more
specifically to random variables indexed by $\mathbb{Z}_+^2$, the
positive two-dimensional integer lattice points.

\end{abstract}

\fot{60F15, 60G50, 60G60, 40G05}{60F05}{Ces\`aro summation, sums of \iid
random variables, complete convergence, convergence in probability,
almost sure convergence, strong law of large numbers}{Ces\`aro
summation for random fields}

\section{Introduction}
\setcounter{equation}{0}
\markboth{A.\ Gut and U.\ Stadtm\"uller}{Ces\`aro summation for random fields}
Various methods of summation for divergent series have been
studied in the literature; see e.g. \cite{hardy, zygmund}. Several
analogous results have been proved for sums of independent, identically
distributed (i.i.d.) random variables.

The most commonly studied method is \emph{Ces\`aro} summation, which
is defined as follows: Let $\{x_n,\,n \geq 0\} $ be a sequence of real
numbers and set, for $\alpha > -1$,
\bea
\ana =
\frac{(\alpha+1)(\alpha+2)\cdots(\alpha+n)}{n!},\quad n = 1,2,\dots,
\ttt{and} A^\alpha_0 = 1.\label{ana}
\eea
The sequence $\{x_n,\, n \geq 0\} $ is said to be $(C,\alpha)$-\emph{summable}
iff
\bea
\frac{1}{\ana}\sum_{k=0}^n \ankame x_k\ttt{converges as}\nifi.
\label{ceskonv}
\eea
It is easily checked (with $A_n^{-1} = 0$ for $n \geq 1$
and $A_0^{-1}=1$) that $(C,0)$-convergence is the same as ordinary
convergence, and that $(C,1)$-convergence is the same as convergence
of the arithmetic means.

Now, let $\{X_k,\, k\geq1\}$ be \iid random variables with
partial sums $\{S_n,\,n\geq1\}$, and let $X$ be a generic random variable.
The following result is a natural probabilistic analog of (\ref{ceskonv}).

\begin{theorem}\label{asconv} Let $0 < \alpha \leq 1$. The sequence
$\{X_k,\, k\geq1\}$ is \emph{almost surely (a.s.)\/}
$(C,\alpha)$-summable iff $E|X|^{1/\alpha} < \infty$. More precisely,
\[
\frac{1}{\ana} \sum_{k=0}^n \ankame X_k \asto\mu\ttt{as}\nifi\ifff
E|X|^{1/\alpha} < \infty \mbox{ and }E\,X=\mu.
\]
\end{theorem}
For $\alpha = 1$ this is, of course, the classical Kolmogorov strong law. For
proofs we refer to \cite{lorentz} ($\frac {1}{2}<\alpha<1$), \cite{cl73}
($0 <\alpha <\frac{1}{2}$) and \cite{dd88} ($\alpha=\frac{1}{2}$).

Convergence in probability for strongly integrable random variables
taking their values in real separable Banach spaces was establised in
\cite{heinkel} under the assumption of strong integrability. In the
real valued case finite mean is not necessary; for $\alpha=1$ we
obtain Feller's weak law of large numbers for which a tail condition is
both necessary and sufficient; cf.\ e.g.\ \cite{g07}, Section 6.4.1.

Next we present Theorem 2.1 of \cite{g93} where complete convergence
was obtained.
\begin{theorem}  Let $0 < \alpha \leq 1$. The sequence
$\{X_k,\, k\geq1\}$ \emph{converges completely to $\mu$},
i.e.,
\[\sumin P\big(\Big|\sum_{k=0}^n  \ankame X_k-\mu\Big|>\ana\varepsilon\big)
<\infty\ttt{for every}\varepsilon>0\,,
\]if and only if
\[\begin{cases} E|X|^{1/\alpha}<\infty,&\ttt{for}0<\alpha<\frac12,\\
 E|X|^2\log^+|X|<\infty,&\ttt{for}\alpha=\frac12,\\
 E|X|^2<\infty,&\ttt{for}\frac12<\alpha\leq 1,
\end{cases}\]
and $E\,X=\mu$.
\end{theorem}
Here and in the following $\log^+x=\max\{\log x,1\}$.

The aim of the present paper is to generalize these results to random fields.
For simplicity we shall focus on random variables indexed by
$\mathbb{Z}_+^2$, leaving the corresponding results for the index set
$\mathbb{Z}_+^d$, $d>2$, to the readers.

The definition of Ces\`aro summability for arrays extends as follows:
\begin{definition} Let $\alpha,\,\beta>0$.
The array $\{x_{m,n},\, m,n \geq 0\} $ is said to be
$(C,\alpha,\beta)$-\emph{summable} iff \bea\label{ceskonv2}
\frac{1}{\ama\anb}\sumsumnull \ankame\anlbme\, x_{k,l}
\ttt{converges as}m,n\to\infty\,. \eea
\end{definition}
Our setup thus is the set $\{\xkl,\, (k,l)\in \mathbb{Z}_+^2\}$ with
partial sums $\smn$, $(m,n)\in \mathbb{Z}_+^2$. The Kolmogorov and
Marcinkiewicz-Zygmund strong law runs as follows.

\begin{theorem} \label{thmzwei}Let $0<r<2$, and suppose that
$X, \{X_{\bk},\,\bk\in \mathbb{Z}^d\}$ are \iid random variables
with partial sums $S_{\bn} =\sum_{\bk\leq \bn}X_{\bk}$, $\bn\in
\mathbb{Z}^d$. If $E|X|^r(\log^+|X|^{d-1})<\infty$, and $E\,X=0$
when $1\leq r<2$, then
\[\frac{S_{\bn}}{|\bn|^{1/r}}\asto0 \ttt{as}
\bn\to\infty.\] Conversely, if almost sure convergence holds as
stated, then $E|X|^r(\log^+|X|^{d-1})<\infty$, and $E\,X=0$ when
$1\leq r<2$.
\end{theorem}
Here $|\bn|=\prod_{k=1}^dn_i$ and  $\bn\to\infty$ means
$\inf_{1\leq k\leq d}n_i\to\infty$, that is, all coordinates tend to infinity.
The theorem was proved in \cite{smythe} for the case $r=1$ and, generally, in
\cite{g78}.

For the analogous weak laws a finite moment of order $r$ suffices (in
fact, even a little less), since convergence in probability is
independent of the order of the index set.

The central object of investigation in the present paper is
\bea\label{ces2}
\frac1{\ama\anb}\sum_{k,l=0}^{m,n}\amkame\anlbme \xkl\,,
\eea
for which we shall establish conditions for convergence in
probability, almost sure convergence and complete convergence

Let us already at this point observe that for $\alpha=\beta=1$ the
quantity in (\ref{ces2}) reduces to that of Theorem \ref{thmzwei} with
$r=1$, that is, to the multiindex Kolmogorov strong law obtained in
\cite{smythe}.  A second thought leads us to extensions of Theorem
\ref{thmzwei} to the case when we do not normalize the partial sums with
the product of the coordinates raised to some power, but the product
of the coordinates raised to \emph{different\/} powers, viz., to, for
example ($d=2$),
\[\frac{\smn}{m^{\alpha}n^{\beta}}\,\ttt{for}0<\alpha<\beta\leq1,\]
(where thus the case $\alpha=\beta=1/r$ relates to Theorem \ref{thmzwei}).
Here we only mention that some surprises occur depending on the domain of the
parameters $\alpha$ and $\beta$. For details concerning this ``asymmetric''
Kolmogorov-Marcinkiewicz-Zygmund extension we refer to \cite{gs3}.

After some preliminaries we present our results for the different
modes of convergence mentioned above. A final appendix contains a
collection of so-called elementary but tedious calculations.

\section{Preliminaries}\label{prel}
\setcounter{equation}{0} Here we collect some facts that will be
used on and off, in general without specific reference.

\noindent
\bulq\quad The first fact we shall use is that whenever
weak forms of convergence or sums of probabilites are inyvolved we may
equivalently compute sums ``backwards'', which, in view of the \iid
assumption shows that, for example
\bea\label{potens}
\sumsumnull P(\amkame\anlbme |X_k|>\ama\anb)<\infty\iff
\sumsum P(\akame\albme|X|>\ama\anb)<\infty.\eea
In the same vein the order of the index set is irrelevant,
that is, one-dimensional results and methods remain valid.

\noindent
\bulq\quad Secondly we recall from (\ref{ana}) that $A^\alpha_0 = 1$ and that.
\[\ana =
\frac{(\alpha+1)(\alpha+2)\cdots(\alpha+n)}{n!},\quad n = 1,2,\dots,
\]
which behaves asymptotically as \bea\label{ana2} \ana \sim
\frac{n^\alpha}{\Gamma(\alpha+1)}\ttt{as}n\to\infty, \eea where
$\sim$ denotes that the limit as $\nifi$ of the ratio between the
members on either side equals 1. Combining the two relations above
tells us that \bea\label{equiv} \sumsumnull
P(\amkame\anlbme|X|>\ama\anb)<\infty\ifff \sumsum
P(k^{\alpha-1}l^{\beta-1}|X|>m^{\alpha}n^{\beta}) <\infty\,. \eea
\bulq\quad We shall also make abundant use of the fact that if
$\{a_k\in\mathbb{R}$, $k\geq1\}$, then \bea\label{mv}
a_n\to0\ttt{as}\nifi\quad\Longrightarrow\quad\frac1n\sumk
a_k\to0\ttt{as}\nifi, \eea
that if, in addition,
$w_k\in\mathbb{R}^+$, $k\geq1$, with $B_n=\sumk w_k$, $n\geq1$,
where $B_n\nearrow\infty$ as $\nifi$, then \bea\label{wmv}
\frac1{B_n}\sumk w_ka_k\to0\ttt{as}\nifi, \eea as well as integral
versions of the same.

\section{Convergence in probability}\label{p}
\setcounter{equation}{0}

We thus begin by investigating convergence in probability.
We do not aim at optimal conditions, except that, as will be seen,
the weak law does not require finiteness of the mean (whereas
the strong law does so).

\begin{theorem}\label{thmprob} Let $0<\alpha\leq\beta\leq1$ and suppose that
$\{\xkl,\,k,l\geq 0\}$ are \iid random variables. Further, set, for
$0\leq k\leq m,\,0\leq l\leq n,$
\[
Y_{k,l}^{m,n}=\amkame\anlbme\xkl I\{|\xkl|\leq \ama\anb\},\quad
S_{m,n}'=\ssumnull Y_{k,l}^{m,n}\ttt{and}\mu_{m,n}=E\,S_{m,n}'.
\]
Then
\bea\label{prob}
\frac1{\ama\anb}\Big(\ssumnull\amkame\anlbme \xkl-\mu_{m,n}\Big)
\pto0\ttt{as}m,n\to\infty
\eea
if
\bea\label{iffprob}
n P(|X|>n)\to0\ttt{as}\nifi\,.
\eea
If, in addition,
\bea\label{mumn}
\frac{\mu_{m,n}}{\ama\anb}\to0\ttt{as}m,n\to\infty,
\eea
then
\bea\label{prob2}
\frac1{\ama\anb}\ssumnull\amkame\anlbme \xkl\pto0
\ttt{as}m,n\to\infty\,.
\eea
\end{theorem}
\begin{remark} \emph{Condition (\ref{iffprob}) is short of $E|X|<\infty$,
i.e., the theorem extends the Kolmogorov-Feller weak law \cite{k28},
\cite{k30}, and \cite{FII}, Section VII.7, to a weak law for weigthed
random fields for a class of weights decaying as powers of order
less than 1 in each direction.}
\end{remark}

\begin{corollary}\label{mean} If, in addition, $E\,X=0$,
then (\ref{prob2}) holds (and if the mean $\mu$ is not equal to zero
the limit in (\ref{prob2}) equals $\mu$).
\end{corollary}
\begin{corollary}\label{symm} If, in addition, the distribution of
the summands is symmetric, then (\ref{iffprob}) alone suffices for
(\ref{prob2}) to hold.\end{corollary}

\pf{Theorem \ref{prob}}
The proof of the theorem amounts to an application of the so-called
degenerate convergence criterion, see e.g. \cite{g07}, Theorem 6.3.3.

Recalling (\ref{potens}) and (\ref{equiv}) we may, equivalently,
prove the theorem for the respective powers of $k$ and $l$, viz., we
redefine the truncated means as
\bea\label{trunc}
Y_{k,l}^{m,n}=k^{\alpha-1}l^{\beta-1}\xkl
 I\{k^{\alpha-1}l^{\beta-1}\,|\xkl|\leq m^{\alpha}n^{\beta}\},
\eea
with partial sums and means as
\bea\label{truncsum}
S_{m,n}'=\ssum Y_{k,l}^{m,n}\ttt{and} \mu_{m,n}=E\,S_{m,n}'\,,
\eea
respectively.

In order to check the conditions of the degenerate convergence criterion
we thus wish to show that, if (\ref{iffprob}) is satisfied, then
\bea\label{wlln1} \ssum P(k^{\alpha-1}l^{\beta-1}|X|> m^\alpha
n^\beta)\to0 \ttt{as}m,n\to\infty\,,
\eea
and that
\bea\label{wlln2}
\frac1{m^{\alpha}n^{\beta}}\ssum\var\big(Y_{k,l}^{m,n}\big)\to0
\ttt{as}m,n\to\infty.
\eea
As for (\ref{wlln1}),
\[\ssum P(k^{\alpha-1}l^{\beta-1}|X|> m^\alpha n^\beta)
= \frac{1}{m^{\alpha}n^{\beta}}\ssum k^{\alpha-1}l^{\beta-1}
\cdot m^{\alpha}n^{\beta}k^{1-\alpha}l^{1-\beta}
P(k^{\alpha-1}l^{\beta-1}|X|> m^\alpha n^\beta),\]
which converges to 0 as $m,n\to\infty$ via (\ref{wmv}).

In order to verify (\ref{wlln2}) we apply the usual
``slicing device'' to obtain
\beaa
&&\hskip-2pc
\frac{1}{m^{2\alpha}n^{2\beta}}\ssum\var\big(Y_{k,l}^{m,n}\big)
\leq\frac1{m^{2\alpha}n^{2\beta}}\ssum  E\big(Y_{k,l}^{m,n}\big)^2\\
&&\hskip2pc\leq\frac{1}{m^{2\alpha}n^{2\beta}}
\ssum E\big(k^{2(\alpha-1)} l^{2(\beta-1)}X^2
I\{k^{\alpha-1}l^{\beta-1}|X|\leq m^\alpha n^\beta\}\big)\\
&&\hskip2pc=\frac{1}{m^{2\alpha}n^{2\beta}}
\ssum k^{2(\alpha-1)} l^{2(\beta-1)}
 \sum_{j=1}^{mn^{\beta/\alpha}}E\big(X^2
I\{(j-1)^\alpha<k^{\alpha-1}l^{\beta-1}|X|\leq j^\alpha\}\big)\\
&&\hskip2pc\leq\frac{1}{m^{2\alpha}n^{2\beta}} \ssum
\sum_{j=1}^{mn^{\beta/\alpha}}j^{2\alpha}\,
P\big((j-1)^\alpha <k^{\alpha-1}l^{\beta-1} |X|\leq j^\alpha \big)\\
&&\hskip2pc\leq\frac{C}{m^{2\alpha}n^{2\beta}} \ssum
\sum_{j=1}^{mn^{\beta/\alpha}}\Big(\sum_{i=1}^j i^{2\alpha-1}\Big)
P\big((j-1)^\alpha <k^{\alpha-1}l^{\beta-1} |X|\leq j^\alpha \big)\\
&&\hskip2pc\leq\frac{C}{m^{2\alpha}n^{2\beta}} \ssum
\sum_{i=1}^{mn^{\beta/\alpha}} i^{2\alpha-1}\,P(|X|\geq
i^\alpha k^{1-\alpha}l^{1-\beta})\\
&&\hskip2pc=\frac{C}{m^{\alpha}n^{\beta}} \ssum
k^{\alpha-1}l^{\beta-1}\Big(\frac{1}{m^{\alpha}n^{\beta}}
\sum_{i=1}^{mn^{\beta/\alpha}}
i^{\alpha-1}\big(i^\alpha k^{1-\alpha}l^{1-\beta}\,P(|X|\geq
i^\alpha k^{1-\alpha}l^{1-\beta})\big)\Big),\\
&&\hskip2pc\to0\ttt{as}m,n\to\infty\,,
\eeaa
by applying (\ref{wmv}) twice to (\ref{iffprob}).
This completes the proof of (\ref{prob}), from which (\ref{prob2}) is
immediate.\vsb

\pf{Corollary \ref{mean}} In order to conclude that also (\ref{prob2})
holds we use the usual method to show that the normalized trruncated
means tend to zero, where w.l.o.g.\ we assume that $E\,X=0$. Then
\beaa
&&\hskip-4pc\Big|\frac1{m^{\alpha}n^{\beta}}\ssum
E\big(k^{(\alpha-1)}l^{(\beta-1)}X
I\{k^{(\alpha-1)}l^{(\beta-1)}|X|\leq m^{\alpha}n^{\beta}\}\big)\Big|\\[4pt]
&=&\Big|-\frac1{m^{\alpha}n^{\beta}}\ssum
E\big(k^{(\alpha-1)}l^{(\beta-1)}X
I\{k^{(\alpha-1)}l^{(\beta-1)}|X|> m^{\alpha}n^{\beta}\}\big)\Big|\\[4pt]
&\leq& \frac1{m^{\alpha}n^{\beta}}\ssum
E\big(k^{(\alpha-1)}l^{(\beta-1)}|X|
I\{k^{(\alpha-1)}l^{(\beta-1)}|X|> m^{\alpha}n^{\beta}\}\big)
\to0\ttt{as}n,m\to\infty.
\eeaa

\pf{Corollary \ref{symm}}
Immediate, since the truncated means are (also) equal to zero.\vsb

\section{Complete convergence}\label{c}
\setcounter{equation}{0}

\begin{theorem}\label{thmc} Let $0<\alpha\leq\beta\leq1$. The field
$\{\xkl,\,k,l\geq 0\}$ \emph{converges completely to $\mu$},
i.e.,
\[\sum_{m\,n} P\big(\Big|\ssumnull \amkame\anlbme \xkl-\mu\Big|
 >\ama\anb\varepsilon\big)
<\infty\ttt{for every}\varepsilon>0\,,
\]
if and only if \goodbreak
\[\begin{cases}
E|X|^{\frac1{\alpha}},&\ttt{for}0<\alpha<1/2\,,\; \alpha<\beta\leq 1,\\[6pt]
E|X|^{\frac1{\alpha}}\log^+|X|,&\ttt{for}0<\alpha=\beta<\frac12,\\[6pt]
E|X|^2(\log^+|X|)^3,&\ttt{for}\alpha=\beta=\frac12,\\[6pt]
E|X|^2(\log^+|X|)^2,&\ttt{for}\alpha=\frac12<\beta\leq 1,\\[6pt]
E|X|^2\log^+|X|,&\ttt{for}\frac12<\alpha\leq\beta\leq 1.\end{cases}
\]
and $E\,X=\mu$.
\end{theorem}  \goodbreak
\begin{proof} For the proof of the sufficiency we refer to the Appendix.

As for the necessity, we argue as in \cite{g92}, p.\ 59. We first suppose
that the distribution is symmetric. Now,
if complete convergence holds, then, using the fact that
\[\max_{0\leq k,l\leq m,n}\amkame\anlbme |\xkl|\leq
2 \max_{0\leq \mu,\nu \leq m,n}\Big|\sum_{k,l=0}^{\mu,\nu}
\amkame\anlbme \xkl\Big|,\] together with the L\'evy inequalities
we must have, say,
\[\sum_{m,n} P\big(\max_{0\leq k,l\leq m,n}\amkame\anlbme |\xkl|
>\ama\anb\big)<\infty\,,
\]
so that, by the first Borel-Cantelli lemma
\[P(\amkame\anlbme|\xkl|>\ama\anb\quad\mbox{i.o. for }
1\leq k,l\leq m,n\,\,; m,n \ge 1)=0.
\]
At this point we use  a device from \cite{nerman}, p.\ 379.
Namely, if the sums $\ssum \amkame\anlbme \xkl$
were independent, we would conclude that
$\sumsum P(\amkame\anlbme |X|>\ama\anb)$ were
finite. Since, however, finiteness of the sum is only a matter of
the tail probabilities, the sum is also finite in the general
case.

An application of (\ref{dreid}) now tells us that the finiteness of
the sum is equivalent to the moment conditions as given in the
statement of the theorem.

This proves the necessity in the symmetric case. The general case
follows the standard desymmetrization procedure, for which we use
Theorem \ref{thmprob} in order to take care of the asymptotics for the
normalized medians (cf.\ \cite{g78}, p.\ 472 for analogous details in the
multiindex setting of the Marcinkiewicz-Zygmund strong laws).\vsb
\end{proof}

\section{Almost sure convergence}\label{as}
\setcounter{equation}{0}
\begin{theorem}\label{thmas} Let $0 < \alpha\leq \beta \leq 1$.  The field
$\{\xkl,\,k,l\geq 0\}$ is \emph{almost surely (a.s.)\/}
$(C,\alpha,\beta)$-summable, that is,
\[
\frac{1}{\ama\anb} \ssumnull\amkame\anlbme \xkl \asto\mu\ttt{as}n,m\to\infty\]
if and only if
\[\begin{cases}
E|X|^{\frac1{\alpha}},&\ttt{for}0<\alpha < \beta\leq 1,\\[6pt]
E|X|^{\frac1{\alpha}}\log^+|X|,&\ttt{for}0<\alpha=\beta\leq1.
\end{cases}
\]
and $E\,X=\mu$.
\end{theorem}
\begin{proof} Since complete convergence always implies almost sure
convergence, the sufficiency follows immediately for the case $\alpha<1/2$.

Thus, let in the following $1/2\leq \alpha\leq \beta\leq 1$. We
first consider the symmetric case (and recall Section \ref{prel}.
In analogy with \cite{heinkel}, p.\ 538, the moment assumptions
permit us to choose an array $\{\etakl,\,k,l\geq 1\}$ of
nonincreasing reals in $(0,1)$ converging to 0, and such that
\[
\sum_{k,l=1}^\infty P(|\xkl|>\etakl k^\alpha l^\beta)<\infty.
\]
Defining
\[Y_{k,l}=\xkl I\{|\xkl|\leq \etakl k^{\alpha}l^{\beta}\}\ttt{and}
S_{m,n}'=\ssumnull Y_{k,l}^{m,n}\,,\]
it thus remains to prove the theorem for the truncated sequence.

This will be achieved via the multiindex Kolmogorov convergence
criterion (see e.g \cite{gab77}) and the multiindex Kronecker lemma
(cf.\ \cite{mo66}).  The first series has just been taken care of, the
second one vanishes since we are in the symmetric case, so it remains
to check the third series.

Toward that end, let, for $k,l\geq1$,
\[Z_{k,l}=\frac{(m-k)^{\alpha-1}(n-l)^{\beta-1}}{m^{\alpha}n^{\beta}}Y_{k,l}\,.
\]
Then
\bea\label{zbdd}
|Z_ {k,l}|\leq \frac{(m-k)^{\alpha-1}(n-l)^{\beta-1}}{m^{\alpha}n^{\beta}}
k^{\alpha} l^{\beta}\etakl\leq \etakl\leq \eta_{00}.
\eea
Now, for any $\varepsilon>0$, arbitrarily small, we may choose our
$\eta$-sequence such that $\eta_{00}<\varepsilon$, so that an application of
the (iterated) Kahane-Hoffman-J\o rgensen inequality (cf.\
\cite{g07}, Theorem 3.7.5) yields
\beaa
P\Big(\Big|\ssumnull Z_{k,l}\Big|>3^j\varepsilon\Big)
&\leq& C_j\bigg(P\Big(\Big|\ssumnull Z_{k,l}\Big|>\varepsilon\Big)\bigg)^{2^j}
\\[4pt] &\leq& C_j
\bigg(\frac{\ssumnull \big((m-k)^{(\alpha-1)}(n-l)^{\beta-1}\big)^{1/\alpha}
E|X|^{1/\alpha}}
{\big(\varepsilon m^{\alpha}n^{\beta}\big)^{1/\alpha}}\bigg)^{2^j}\\[4pt]
&=&C_j'\bigg(\frac{\ssumnull k^{(1-1/\alpha)} l^{(\beta-1)/\alpha}}
{mn^{\beta/\alpha}}\bigg)^{2^j}\\
&=&\begin{cases}
C_j''\Big(\frac1{(mn)^{\frac1{\alpha}-1}}\Big)^{2^j},
&\ttt{for}\frac12<\alpha<\beta<1,\\[6pt]
C_j''\Big(\frac{\log m}{nm}\Big)^{2^j},&\ttt{for}\frac12=\alpha<\beta<1,\\[6pt]
C_j''\Big(\frac{\log m\log n}{nm}\Big)^{2^j},
&\ttt{for}\frac12=\alpha=\beta,\end{cases}
\eeaa
(since the usual first term in the RHS vanishes in view of (\ref{zbdd})).

By choosing $j$ sufficiently large it then follows that
\[\sumsumnull P\Big(\Big|\ssumnull Z_{k,l}\Big|>3^j\varepsilon\Big)<\infty.\]
By replacing $3^j\varepsilon$ by $\varepsilon$ we have thus, due to the
arbitrariness of $\varepsilon$, shown that
\bea
P\big(|Z_ {k,l}|>\varepsilon \mbox{ i.o.}\big)=0 \ttt{for any}\varepsilon>0,
\eea
from which the desired almost sure convergence follows via the multiindex
Kronecker lemma referred to above.

This proves the sufficiency in the symmetric case from which
the general case follows by the standard desymmetrization procedure
hinted at in the proof of Theorem \ref{thmc}.

Finally, suppose that almost sure convergence holds as stated.
It then follows that
\[
\frac{A_0^{\alpha-1}A_0^{\beta-1}X_{m,n}}{\ama\anb}\asto0\ttt{as}m,n\to\infty,
\]
and, hence, also that
\[
\frac{X_{m,n}}{m^{\alpha}n^{\beta}}\asto0\ttt{as}m,n\to\infty,
\]
which, in view of \iid assumption and the second Borel-Cantelli lemma,
tells us that
\[
\sum_{m,n} P(|X|>m^{\alpha}n^{\beta})<\infty,
\]
which, in turn, is equivalent to the given moment conditions.

This concludes the proof of the theorem.\vsb
\end{proof}

\section{Concluding remarks}

We close with some comments on the present and related work.
\begin{itemize}
\item Convergence in probability has earlier been established in
\cite{heinkel} via approximation with indicator variables, and under the
assumption of finite mean. Our proof is simpler (more elementary)
and presupposes only a Feller condition.
\item As pointed out above, almost sure convergence was established in
three steps (\cite{lorentz}, \cite{cl73} and \cite{dd88}) with
different proofs.  Our proof, which also works for the case $d=1$,
takes care of the whole proof in one go (since our proof also works for the
case $\alpha<1/2$).
\item For simplicity we have confined ourselves to the case $d=2$. The
same ideas can be modified for the case $d>2$ and
$(C,\alpha_1,\alpha_2,\ldots,\alpha_d)$-summability. However, the
moment conditions then depend on the number of $\alpha$:s that are
equal to the smallest one (corresponding to $\alpha<\beta$ or
$\alpha=\beta$ in the present paper); see \cite{gs3} for
Kolmogorov-Marcinkiewicz-Zygmund laws.
\item Results on complete convergence are special cases of results on
convergence rates. In this vein our results are extendable to results
concerning
\[\sum_{m,n} n^{r-2}m^{r-2}P\big(\Big|\ssumnull \amkame\anlbme \xkl-\mu\Big|
 >\ama\anb\varepsilon\big) <\infty\ttt{for every}\varepsilon>0\]
(cf.\ \cite{g93} for the case $d=1$). For the proofs one would need i.a.\
extensions of the relevant computations in the appendix below.
\end{itemize}

\renewcommand{\theequation}
{A.\arabic{equation}}

\appendix
\section{Appendix}\label{app}
\setcounter{equation}{0} In this appendix we collect a number of
so-called elementary but tedious calculations.

First, let $0<\alpha\leq \beta<1$. Then \bea &&\sumsum
P(k^{\alpha-1}l^{\beta-1}|X|>m^{\alpha}n^{\beta})
<\infty\ifff\nonumber \\[4pt]
&&\int_1^\infty\int_1^\infty\int_1^x\int_1^y
P(|X|>u^{1-\alpha}v^{1-\beta}x^{\alpha}y^{\beta})\,dudvdxdy<\infty\ifff
\nonumber \\[4pt]
&&\hskip2cm\Big[u^{1-\alpha}x^{\alpha}=z,\qquad
v^{1-\beta}y^{\beta}=w \Big]
\nonumber\\[6pt]
&&\int_1^\infty\int_1^\infty\int_{x^{\alpha}}^x\int_{y^{\beta}}^y
\Big(\frac{z}{x}\Big)^{\amaa}\Big(\frac{w}{y}\Big)^{\bnbb}
P(|X|>zw)\,dzdwdxdy<\infty\ifff\nonumber \\[4pt]
&&\int_1^\infty\int_1^\infty
\bigg(\int_z^{z^{1/\alpha}}\frac{dx}{x^{\amaa}}\bigg)
\bigg(\int_w^{w^{1/\beta}}\frac{dy}{y^{\bnbb}}\bigg)z^{\amaa}
w^{\bnbb} P(|X|>zw)\,dzdw<\infty\,.\label{dreia} \eea
In case
$0<\alpha<\beta=1$ we have
\bea &&\sumsum P(k^{\alpha-1}|X|>m^{\alpha}n)<\infty\ifff\nonumber \\[4pt]
&&\int_1^\infty\int_1^\infty
\bigg(\int_z^{z^{1/\alpha}}\frac{dx}{x^{\amaa}}\bigg) z^{\amaa}
\,w \,P(|X|>zw)\,dzdw<\infty\,.\label{dreia1} \eea
 Next we note that
\bea\label{dreib}
\int_y^{y^{1/\gamma}}\frac{dx}{x^{\gmgg}}\sim C\,
\begin{cases} y^{\frac{1-2\gamma}{\gamma(1-\gamma)}},
&\ttt{for}0<\gamma<\frac12,\\
\log y,&\ttt{for}\gamma=\frac12,\\
y^{\frac{1-2\gamma}{1-\gamma}},&\ttt{for}\frac12<\gamma<1,
\end{cases}
\eea so that
\beaa
&&\hskip-2pc\bigg(\int_z^{z^{1/\alpha}}\frac{dx}{x^{\amaa}}\bigg)
\bigg(\int_w^{w^{1/\beta}}\frac{dy}{y^{\bnbb}}\bigg)z^{\amaa}w^{\bnbb}\\[8pt]
&&\hskip2pc\sim C\,\begin{cases}
 z^{\amaaa}w^{\bnbbb},&\ttt{for}0<\alpha,\beta<\frac12,\\
(zw)^{\amaaa},&\ttt{for}0<\alpha=\beta<\frac12,\\
zw\log z\log w=\frac{zw}{2}\big((\log zw)^2&\\
\hskip2pc-(\log z)^2-(\log w)^2\big),
&\ttt{for}\alpha=\beta=\frac12,\\
z^{\amaaa}w\log w,&\ttt{for}\alpha<\beta=\frac12,\\
z^{\amaaa}w,&\ttt{for}\alpha<\frac12<\beta\leq1,\\
zw\log z,&\ttt{for}\alpha=\frac12<\beta\leq1,\\
zw,&\ttt{for}\frac12<\alpha\leq\beta\leq1,
\end{cases}
\eeaa from which it follows that \bea
&&\hskip-2pc\int_1^\infty\int_1^\infty
\bigg(\int_z^{z^{1/\alpha}}\frac{dx}{x^{\amaa}}\bigg)
\bigg(\int_w^{w^{1/\beta}}\frac{dy}{y^{\bnbb}}\bigg)z^{\amaa}
x^{\bnbb}
P(|X|>zw)\,dzdw\nonumber\\[4pt]
&&\hskip3cm=\Big[x=zw,\qquad y=z \Big]\nonumber\\[6pt]
&&=\begin{cases} \int_1^\infty\int_1^x
x^{\bnbbb}y^{\frac1{\alpha}-\frac1{\beta}-1}
P(|X|>x)\,dydx\\[4pt]\hskip2pc
=C\int_1^\infty x^{\frac1{\alpha}-1}P(|X|>x)\,dx,
&\ttt{for}0<\alpha<\beta<\frac12,\\[6pt]
\int_1^\infty\int_1^x x^{\amaaa}\frac{1}{y}P(|X|>x)\,dydx\\[4pt]\hskip2pc
=C\int_1^\infty x^{\amaaa}\log xP(|X|>x)\,dx,
&\ttt{for}0<\alpha=\beta<\frac12,\\[6pt]
\int_1^\infty\int_1^x \big(\frac12x(\log
x)^2\frac1{y}-x\frac{(\log y)^2}{y}
\big)P(|X|>x)\,dxdy&\\[4pt]\hskip2pc=\frac16\int_1^\infty x(\log x)^3
P(|X|>x)\,dx,
&\ttt{for}\alpha=\beta=\frac12,\\[6pt]
\int_1^\infty\int_1^x xy^{\frac1{\alpha}-2}(\log x-\log y)P(|X|>x)\,dydx
&\\[4pt]\hskip2pc=C\int_1^\infty x^{\frac1{\alpha}-1}P(|X|>x)\,dx,
&\ttt{for}\alpha<\beta=\frac12,\\[6pt]
\int_1^\infty\int_1^x x y^{\frac1{\alpha}-2}P(|X|>x)\,dydx\\[4pt]\hskip2pc
=C\int_1^\infty x^{\frac1{\alpha}-1}P(|X|>x)\,dx,
&\ttt{for}\alpha<\frac12<\beta\leq 1,\\[6pt]
\int_1^\infty\int_1^x x\frac{\log y}{y}P(|X|>x)\,dydx\\[4pt]\hskip2pc
=\frac12\int_1^\infty x(\log x)^2P(|X|>x)\,dx,
&\ttt{for}\alpha=\frac12<\beta\leq 1,\\[6pt]
\int_1^\infty\int_1^x x\frac1{y}P(|X|>x)\,dydx\\[4pt]\hskip2pc
=\frac12\int_1^\infty x\log x P(|X|>x)\,dx,
&\ttt{for}\frac12<\alpha\leq\beta\leq 1.
\end{cases}\label{dreic}
\eea

Summarizing this we have shown that, for $0<\alpha\leq \beta<1$,
\bea
&&\hskip-2pc\sumsum P(\akame\albme|X|>\ama\anb)<\infty\ifff \\[4pt]
&&\hskip4pc\begin{cases}
E|X|^{\frac1{\alpha}},&\ttt{for}0<\alpha<1/2,\,\alpha<\beta\leq 1,
\\[6pt]
E|X|^{\frac1{\alpha}}\log^+|X|,&\ttt{for}0<\alpha=\beta<\frac12,\\[6pt]
E|X|^2(\log^+|X|)^3,&\ttt{for}\alpha=\beta=\frac12,\\[6pt]
E|X|^2(\log^+|X|)^2,&\ttt{for}\alpha=\frac12<\beta\leq 1,\\[6pt]
E|X|^2\log^+|X|,&\ttt{for}\frac12<\alpha\leq\beta\leq
1.\end{cases}\label{dreid} \eea

\subsection*{Acknowledgement} The work on this paper has been
supported by Kungliga Vetenskapssamh\"allet i Uppsala. Their support
is gratefully acknowledged. In addition, the second author likes to
thank his partner Allan Gut for the great hospitality during two
wonderful and stimulating weeks at the University of Uppsala.

\medskip\noindent {\small Allan Gut, Department of Mathematics,
Uppsala University, Box 480,\\ SE-751\,06 Uppsala, Sweden;\\
Email:\quad \texttt{allan.gut@math.uu.se}\\
URL:\quad \texttt{http://www.math.uu.se/\~{}allan}}
\\[4pt]
{\small Ulrich Stadtm\"uller, Ulm University, Department of Number
Theory
and Probability Theory,\\ D-89069 Ulm, Germany;\\
Email\quad \texttt{ulrich.stadtmueller@uni-ulm.de}\\
URL:\quad
\texttt{http://www.mathematik.uni-ulm.de/matheIII/members/stadtmueller/stadtmueller.html}}
\end{document}